\documentclass[12pt]{amsart}
\usepackage{amsmath,amssymb,amsthm,mathscinet}
\usepackage{tabularx}
\usepackage{hyperref}
\newtheorem{thm}{Theorem}[section]
\newtheorem{lem}[thm]{Lemma}

\theoremstyle{definition}
 
\theoremstyle{remark}

\newcommand{\nequiv}{\not\equiv}

\let\lm=\lambda
\newcommand{\mathmod}[1]{\ \left(\mathrm{mod}\ #1\right)}

\newcommand{\acr}{\newline\indent}

\begin{document}

\title{On a problem of De Koninck}
\author[Tomohiro Yamada]{Tomohiro Yamada*}
\address{\llap{*\,}Center for Japanese language and culture\acr
                   Osaka University\acr
                   562-8678\acr
                   3-5-10, Sembahigashi, Minoo, Osaka\acr
                   JAPAN}
\email{tyamada1093@gmail.com}


\begin{abstract}
Let $\sigma(n)$ and $\gamma(n)$ denote the sum of divisors and
the product of distinct prime divisors of $n$ respectively.
We shall show that, if $n\neq 1, 1782$ and $\sigma(n)=(\gamma(n))^2$,
then there exist odd (not necessarily distinct) primes $p, p^\prime$ and (not necessarily odd)
distinct primes $q_i (i=1, 2, \ldots, k)$ such that
$p, p^\prime\mid\mid n$, $q_i^2\mid\mid n (i=1, 2, \ldots, k)$ with $k\leq 3$ and
$q_1\mid \sigma(p^2), q_{i+1}\mid\sigma(q_i^2) (1\leq i\leq k-1), p^\prime \mid\sigma(q_k^2)$.
\end{abstract}

\maketitle

\section{Introduction}

Let $\sigma(n)$ and $\gamma(n)$ denote the sum of divisors and
the product of distinct prime divisors of $n$, called the {\em radical} of $n$, respectively.
Moreover, let $\omega(n)$ denote the number of distinct prime divisors of $n$.
De Koninck \cite{Kon} posed the problem to prove or disprove that
the only solutions of
\begin{equation}\label{eq10}
\sigma(n)=(\gamma(n))^2
\end{equation}
are $n=1, 1782$.

According to the editorial comment, it is shown that such an integer $n\neq 1, 1782$ must be even,
have at least four prime factors, be neither square-free nor squarefull,
be greater than $10^9$ and have no prime factor raised to a power congruent to $3\mathmod{4}$.
Later, further necessary conditions to satisfy \eqref{eq10} have been shown.
Broughan, De Koninck, K\'{a}tai and Luca \cite{BKKL} showed that,
if an integer $n>1$ satisfies \eqref{eq10}, then
\begin{equation}\label{eq11}
n=2^{e_0} \prod_{i=1}^s p_i^{e_i},
\end{equation}
where $p_i$ are distinct odd primes and $e_i$ are positive integers satisfying
(a) $p_1\equiv 3\mathmod{8}, e_1=1$ and the other $e_i$'s are even,
or
(b) $p_1\equiv p_2\equiv e_1\equiv e_2\equiv 1\mathmod{4}, \min\{e_1, e_2\}=1$ and the other $e_i$'s are even.
Moreover, they showed that $\omega(n)\geq 5$ and $n$ cannot be fourth power free.

Broughan, Delbourgo and Zhou \cite{BDZ} showed that $p_1\geq 43$ in the case (a),
$p_1\geq 173$ in the case (b) with $\alpha_2>\alpha_1=1$
and $n$ must be divisible by the fourth power of an odd prime.
Chen and Tong \cite{CT} showed that if $n\neq 1, 1782$ satisfies \eqref{eq10} with (a),
then $n$ is divisible by $3$ and by the fourth powers of at least two odd primes, $p_1\geq 1571$,
at most two of $p_i$'s are greater than $p_1$,
$e_i=2$ for at least two $i$'s and $e_i=2$ for any $i$ such that $10p_i^2\geq p_1$.
Moreover, they showed that for any $n$ satisfying \eqref{eq10},
at least half of the numbers among $e_i+1$'s $(0\leq i\leq s)$ must be either primes or prime squares.
Tang and Zhou \cite{TZ} showed that no integer $n=2^{e_0} p_1 p_2 p_3^4 p_4^4$ other than $1, 1782$ satisfies \eqref{eq10}.
Furthermore, there exist only finitely many integers of the form \eqref{eq11} satisfying \eqref{eq10}
for any given integer $s$.
More generally, Luca \cite{Luc} showed that, if a positive integer $n$ satisfies $\omega(n)=T$
and $\sigma(n)\mid L(\gamma(n))^K$ with $K, L$ positive integers, then
\[n<\exp \left(((K+\log L)T!)^{2^T}\right).\]

As usual, $p^e\mid\mid n$ denotes that $p^e\mid n$ but $p^{e+1}\nmid n$.
In this paper, we shall give the following new necessary condition for an integer $n$
to satisfy \eqref{eq10}.

\begin{thm}\label{th11}
If an integer $n\neq 1, 1782$ of the form \eqref{eq11} satisfies \eqref{eq10},
then there exist odd (not necessarily distinct) primes $p, p^\prime$ and (not necessarily odd)
distinct primes $q_i (i=1, 2, \ldots, k)$ with $k\leq 3$ such that
$p, p^\prime \mid\mid n$, $q_i^2\mid\mid n (i=1, 2, \ldots, k)$ and
$q_1\mid \sigma(p^2), q_{i+1}\mid\sigma(q_i^2) (i=1, 2, \ldots, k-1), p^\prime \mid\sigma(q_k^2)$.
\end{thm}

Our idea is based on the following simple observation, which has been used by previous authors
mentioned above.
For example, consider the special case $e_i=1$ only for $i=1$, $q_1\mid \sigma(p^2)$ for two primes $p$ and
for each $p$, $p\mid \sigma(q_i^{e_i})$ with $e_i\geq 4$ for two primes $q_i$.
Now we have $\sigma(q_i^{e_i})/q_i^2>\sqrt{\sigma(q_i^{e_i})}>p^{1/2}>q_1^{1/4}$ for each $i$.
Hence, $((q_1+1)/q_1^2) \prod_i \sigma(q_i^{e_i})/q_i^2>q_1(q_1+1)/q_1^2>1$
and $\sigma(n)/(\gamma(n))^2>1$, which contradicts \eqref{eq10}.
In order to generalize this observation,
we introduce a directed multigraph related to prime power divisors of $n$.

In the next section, we introduce some basic terms on directed multigraphs and prove
an identity on directed multigraphs.
In Section 3, we introduce a certain directed multigraph related to prime power divisors of $n$
satisfying \eqref{eq10} and give the key lemma for our proof
as well as some arithmetic preliminaries.

Under our settings described in Sections 2 and 3, we shall prove the following theorem.

\begin{thm}\label{th12}
Let $n\neq 1, 1782$ be an integer of the form \eqref{eq11} satisfying \eqref{eq10}
and $L$ be the set of odd prime divisors $q_i$'s with $e_i=1$.
Let $G(n), N=N(L), M=M(L), B=B(L)$ and $C=C(L)$ be directed multigraphs or sets defined in Section 3.
Then,
\begin{itemize}
\item[i)] If $q_{k+1}\rightarrow q_k\rightarrow \cdots \rightarrow q_1\rightarrow p$
is a path from a vertex $q_{k+1}$ in $B$ to a vertex $p$ in $L$ via vertices in $M$,
then $k\leq 3$ and $q_i\equiv 1\mathmod{3}$ for $1\leq i\leq k-1$.
\item[ii)] $M$ contains at most two primes $\equiv 1\mathmod{3}$.
Furthermore, $\# M\leq 6$ if $\# L=1$ and $\# M\leq 8$ if $\# L=2$.
\item[iii)] There exists a path from $q_i$ in $L$ to $q_j$ in $L$
consisting of vertices $q_l\in N$ other than $q_i, q_j$,
where $q_i$ and $q_j$ may be the same prime.
\end{itemize}
\end{thm}

Now Theorem \ref{th11} is an arithmetic translation of iii) of Theorem \ref{th12}.
In Sections 4 and 5, we prove that
the directed multigraph related to prime power divisors of $n$ defined in Section 3
cannot have some forms, which yields iii) of Theorem \ref{th12}.
Other statements of Theorem \ref{th12} easily follow from an elementary divisibility property of
values of $\sigma(p^2)$ with $p$ prime.

\section{An identity on directed multigraphs}

Before stating our result on directed multigraphs,
we would like to introduce some basic terms on directed multigraphs according to \cite{BJG}
with some modifications.
A {\em directed multigraph} $G=(V, A)$ consists of a set $V$ of elements called {\em vertices}
and a multiset $A$, where an element may be contained more than once,
of ordered pairs of distinct elements in $V$ called {\em arcs}.
$V=V(G)$ and $A=A(G)$ are called the vertex set and the arc set of $G$ respectively.
For an arc $(u, v)$ in $A$, which we call an arc from $u$ to $v$,
the former vertex $u$ and the latter vertex $v$ are called its {\em tail} and its {\em head} respectively.
We often write $u\rightarrow v$ if $(u, v)\in A$
and $u\overset{k}{\rightarrow} v$ if $(u, v)\in A$ exactly $k$ times.

The {\em subgraph} of $G=(V, A)$ spanned by a given set of vertices $S\subset V$ is
the directed multigraph whose vertex set is $S$ and whose arc set consists of all arcs in $A$
whose tail and head both belong to $S$.

A {\em walk} $(a_1, a_2, \ldots, a_k)$ of length $k$
is a sequence of arcs $a_i=(u_i, v_i) (i=1, 2, \ldots, k)$
such that $v_i=u_{i+1}$ for all $i=1, 2, \ldots, k-1$.
A walk $(a_1, a_2, \ldots, a_k)$ with $a_i=(u_i, u_{i+1}) (i=1, 2, \ldots, k)$
is called a {\em path} if $u_1, u_2, \ldots, u_{k+1}$ are all distinct
and a {\em cycle} if $u_1, u_2, \ldots, u_k$ are all distinct and $u_1=u_{k+1}$.
A walk $(a_1, a_2, \ldots, a_k)$ with $a_i=(u_i, u_{i+1}) (i=1, 2, \ldots, k)$
is often written as $u_1\rightarrow u_2\rightarrow \cdots \rightarrow u_{k+1}$.
A directed multigraph $G=(V, A)$ is called {\em acyclic} if $A$ contains no cycle.

The {\em out-degree} $d^+(v)=d_G^+(v)$ and the {\em in-degree} $d^-(v)=d_G^-(v)$ of the vertex $v$
are the number of arcs from $v$ and to $v$ respectively counted with multiplicity.
A vertex $v$ is called a {\em sink} if $d^+(v)=0$
and a {\em source} if $d^-(v)=0$.
$S(G)$ denotes the set of sources of the directed multigraph $G$.

Now we would like to state our identity.

\begin{lem}\label{lm21}
Let $G$ be a directed acyclic multigraph.
Then, for any vertex $v_0$ of $G$ with $d^-(v_0)>0$,
\begin{equation}\label{eqa1}
\sum_{\substack{P: v_k\rightarrow v_{k-1}\rightarrow\cdots \rightarrow v_0\subset G, \\v_k\in S(G)}} \frac{1}{\prod_{i=0}^{k-1} d_G^-(v_i)}=1.
\end{equation}
\end{lem}

\begin{proof}
If $G$ consists of only one sink $v_0$ and sources $u_1, u_2, \ldots, u_l$ with arcs $(u_i, v_0)$,
then \eqref{eqa1} is clear.

For any fixed vertices $v_0, v_1, \ldots, v_{k-1}$ such that $v_{k-1}\rightarrow v_{k-2}\rightarrow \cdots \rightarrow v_0$ and any vertex $w\rightarrow v_{k-1}$ is a source in $G$, we have
\begin{equation}
\sum_{\substack{v_k\in S(G), \\ v_k\rightarrow v_{k-1}\rightarrow\cdots \rightarrow v_0\subset G}} \frac{1}{\prod_{i=0}^{k-1} d_G^-(v_i)}=\frac{1}{\prod_{i=0}^{k-2} d_G^-(v_i)}.
\end{equation}
Thus, setting $H$ to be the directed multigraph obtained from $G$ by eliminating all arcs to $v_{k-1}$,
we have
\begin{equation}
\sum_{\substack{P: v_k\rightarrow v_{k-1}\rightarrow\cdots \rightarrow v_0\subset G,\\ v_k\in S(G)}} \frac{1}{\prod_{i=0}^{k-1} d_G^-(v_i)}
=\sum_{\substack{P: v_k\rightarrow v_{k-1}\rightarrow\cdots \rightarrow v_0\subset H,\\ v_k\in S(H)}} \frac{1}{\prod_{i=0}^{k-1} d_H^-(v_i)}.
\end{equation}

Since $G$ is acyclic, this descent argument eventually reduces $G$ to a directed multigraph
$(V, A)$ with $V=\{v_0, u_1, u_2, \ldots, u_l\}$ and $A=\{(u_i, v_0), i=1, \ldots, l\}$.
Now the lemma follows by induction.
\end{proof}

\section{A directed multigraph related to divisors of an integer}

Let $n$ be a positive integer greater than one.
We define the directed multigraph $G=G(n)$ arising from $n$
by setting its vertex set to be the set of primes dividing $n\sigma(n)$
and each arc $p \overset{k}{\rightarrow} q$ to be of multiplicity $k$
if $q^k \mid\mid\sigma(p^e)$ for the exponent $e$ with $p^e\mid\mid n$.
For convenience, we write $p^e\rightarrow q^f$ if $p\rightarrow q$ and $p^e, q^f\mid\mid n$
and $p^e\in S$ if $p^e\mid\mid n$ and $p$ belongs to a set $S$ of vertices.

For a set $S$ of vertices $w_1, w_2, \ldots, w_k$ of $G$,
we define their {\em $2$-incomponent} $N(S)$ to be the subgraph of $G$
consisting $w_1, w_2, \ldots, w_k$ themselves and the vertices $w$ such that there exists a path
$v^2\rightarrow v_1^2 \rightarrow \cdots \rightarrow v_l^2 \rightarrow w_i$ to some vertex $w_i$,
their {\em $2$-boundary} $B(S)$ by the set of vertices $v\not\in N(S)$
from which there exists an edge to some vertex in $N(S)$
and their {\em $2$-closure} $C(S)$ by the subgraph whose vertex set is $N(S)\cup B(S)$
and whose arc set consists of all edges in $B(S)$ and all arcs from $N(S)$ to $B(S)$.
For convenience, we simply write $N(w)$ for $N(\{ w\})$ and so on.
Moreover, we put $p_0=2$ and $M(S)=N(S)\backslash S$.
We note that $C(S)$ may contain $p_0=2$.

Now Theorem \ref{th11} can be restated as in iii) of Theorem \ref{th12}.

For a set $S$ of prime powers, we define $h(S)=\prod_{p^e\in S} \sigma(p^e)/p^2$.
Clearly, we have $h(S_0)=\sigma(n)/(\gamma(n))^2$ for the set $S_0$ of all prime-power divisors of $n$.
For convenience, we write $h(p^e)=h(\{p^e\})$ for a prime power $p^e$
and $h(n)=h(S_0)$ for the set $S_0$ mentioned above.

We clearly have the following lemma.

\begin{lem}\label{lm31}
We have $h(m)\geq 1$ for any positive integer $m$ with the equality just when $m=1$.
If $m_1$ divides $m_2$, then $h(m_1)\leq h(m_2)$.
Furthermore, if $S$ and $T$ are disjoint sets of prime-power divisors of $n$, then
$h(S\cup T)=h(S)h(T)$.
\end{lem}

We also use the following divisibility property of values of the polynomial $x^2+x+1$.

\begin{lem}\label{lm32}
If $m$ is an integer and a prime $p$ divides $m^2+m+1$, then $p=3$ or $p\equiv 1\mathmod{3}$.
Furthermore, $3$ divides $m^2+m+1$ if and only if $m\equiv 1\mathmod{3}$.
\end{lem}
\begin{proof}
The former is a special case of Theorem 94 of \cite{Nag}.
Indeed, if $p\neq 3$ divides $m^2+m+1$, then $m\nequiv 1\mathmod{p}$ and $m^3\equiv 1\mathmod{p}$.
Hence, $m\mathmod{p}$ has the multiplicative order $3$ and therefore $p-1$ must be divisible by $3$.
The latter can be easily confirmed by calculating modulo $3$.
\end{proof}

The following lemma is the key point of our proof of Theorem \ref{th11}.

\begin{lem}\label{lm40}
Let $n$ be an integer of the form \eqref{eq11} satisfying \eqref{eq10}
and $L$ be a set of prime power divisors of $n$.
We define quantities $\kappa_i$ for $p_i\in C=C(L)$ and $\lm_i$ for $p_i\in M=M(L)$ by
\begin{equation}
\sigma(p_i^{e_i})=\kappa_i \prod_{p_j\in N(L)} p_j^{k_{i, j}}
\end{equation}
and
\begin{equation}
p_i^2=\lm_i \prod_{p_j\in N(L)} p_j^{k_{i, j}},
\end{equation}
where $\kappa_i, \lm_i$ are integers not divisible by any prime in $N(L)$.

If $N=N(L)$ is acyclic and any element of $L$ is a sink of $N$, then
\begin{equation}
\prod_{p_i\in B}\sigma(p_i^{e_i})=\prod_{p_i\in B}\kappa_i
\prod_{p_j\in M}\lm_j \prod_{p_i\in L}p_i^2
\end{equation}
and
\begin{equation}
h(C)>\prod_{p_i\in B}\kappa_i^\frac{1}{2} p_i^{\frac{e_i}{2}-2}
\prod_{p_j\in M}\frac{\sqrt{\sigma(p_j^2)}}{p_j} \prod_{p_i\in L}p_i^{e_i-1}.
\end{equation}
\end{lem}

\begin{proof}
We see that
\begin{equation}\label{eq41}
p_i=\lm_i^\frac{1}{2} \prod_{p_i\rightarrow p_j, p_j\in N} p_j^{\frac{1}{2}}
\end{equation}
for $p_i\in M$.
Since we assume that a vertex in $L$ must be a sink in $C=C(L)$,
if $P: q_1^2 \rightarrow \cdots \rightarrow q_k^2\rightarrow q_0$ is a path in $N$
and a prime $q$ in $L$ occurs in $P$, then $q=q_0$.
Moreover, by the assumption, $N$ is acyclic.
Hence, we iterate \eqref{eq41} to obtain
\begin{equation}\label{eq42}
q_1=\prod_{q_1^2 \rightarrow \cdots \rightarrow q_k^2\rightarrow p^1, p\in L}
(\lm_{j_1}^\frac{1}{2}\lm_{j_2}^\frac{1}{4}\cdots \lm_{j_k}^\frac{1}{2^k}) q_i^\frac{1}{2^k}
\end{equation}
for any $q_1\in M$, where the $j_m$'s $(m=1, 2, \ldots, k)$ are indices such that $p_{j_m}=q_m$.

Moreover, we see that
\begin{equation}\label{eq43}
\sigma(p_i^{e_i})=\kappa_i \prod_j p_j^{k_{i, j}}
=\kappa_i \prod_{p_i\rightarrow p_j, p_j\in N} p_j
\end{equation}
for $p_i\in B$.
Combining \eqref{eq42} and \eqref{eq43}, we have
\begin{equation}
\prod_{p_i\in B} \sigma(p_i^{e_i})=\left(\prod_{p_i\in B}\kappa_i\right)\prod_{p_j\in M}\lm_j^{s_j}
\prod_{p_i\in L} p_j^{2s_j},
\end{equation}
where, observing that $d_C^-(p_i)=d_G^-(p_i)=2$ for any $p_i\in N$ from \eqref{eq10},
\begin{equation}
s_j=\sum_{\substack{q_0\rightarrow q_1\rightarrow \cdots \rightarrow q_k=p_j, \\
q_0\in B, q_1, \ldots, q_{k-1}\in N}} \frac{1}{2^k}
=\sum_{\substack{q_0\rightarrow q_1\rightarrow \cdots \rightarrow q_k=p_j, \\
q_0\in B, q_1, \ldots, q_{k-1}\in N}} \frac{1}{\prod_{l=1}^k d_C^-(q_l)}.
\end{equation}

Since $N$ is acyclic by the assumption,
Lemma \ref{lm21} gives that $s_j=1$ for all $p_j\in N$.
Thus we obtain
\begin{equation}
\prod_{p_i\in B}\sigma(p_i^{e_i})=\left(\prod_{p_i\in B}\kappa_i\right)\left(\prod_{p_j\in M}\lm_j\right) \prod_{p_i\in L}p_i^2
\end{equation}
and therefore
\begin{equation}
\begin{split}
\prod_{p_i\in C}\frac{\sigma(p_i^{e_i})}{p_i^2}
> & \prod_{p_i\in B} p_i^{\frac{e_i}{2}-2}\sqrt{\sigma(p_i^{e_i})}
\prod_{p_j\in M}\frac{\sigma(p_j^2)}{p_j^2} \prod_{p_i\in L}p_i^{e_i-1} \\
\geq & \prod_{p_i\in B} \kappa_i^\frac{1}{2} p_i^{\frac{e_i}{2}-2}
\prod_{p_j\in M}\lm_j^\frac{1}{2} \frac{\sigma(p_j^2)}{p_j^2} \prod_{p_i\in L} p_i^{e_i-1}.
\end{split}
\end{equation}
Now the lemma immediately follows observing that $\lm_j\geq p_j^2/\sigma(p_j^2)$ for $p_j\in M$.
\end{proof}

\section{Acyclic cases}

In this and the next sections,
we assume that $n$ is an integer of the form \eqref{eq11} satisfying \eqref{eq10}
and we put $L$ to be the set of odd primes $p_i$ with $e_i=1$.
Thus, $L=\{p_1, p_2\}$ in the case (b) with $e_1=e_2=1$ and
$L=\{p_1\}$ in the case (a) and the case (b) with $e_1=1<e_2$.
In this section, we shall show that, 
$N=N(L)$ must have a cycle or we must have $L=\{p_1, p_2\}$ and $p_1\in B(p_2)$ or $p_2\in B(p_1)$.

\begin{lem}\label{lm41}
If $n$ is divisible by $4$ or $2\times 3^6$ or
$n$ is divisible by $2$ and $3$ does not belong to $C=C(L)$, then
$N=N(L)$ must have a cycle or we must have $L=\{p_1, p_2\}$ and $p_1\in B(p_2)$ or $p_2\in B(p_1)$.
\end{lem}

\begin{proof}
Assume that $n$ of the form \eqref{eq11} is divisible by $2^2$ or $2\times 3^6$
and $N$ is acyclic and, in the case $L=\{p_1, p_2\}$, $p_1\not\in B(p_2)$ and $p_2\not\in B(p_1)$.

We can easily see that any prime $p_i$ in $L$ must be a sink in $N$.
Indeed, if $p_i\in L$ and $p_i\rightarrow p_j$ for some $p_j\in N$
not necessarily distinct from $p_i$, then, there exists a path
from $p_i$ to $p_j\in L$ via $N$, which contradicts the assumption.
Thus, we can apply Lemma \ref{lm40} and, observing that $\kappa_i\geq 1$ for all $p_i\in B=B(L)$,
we obtain
\begin{equation}\label{eq44}
h(C)>\prod_{p_i\in B} p_i^{\frac{e_i}{2}-2}.
\end{equation}

If $4=2^2$ divides $n$, then, observing that $e_i/2\geq 2$ for $p_i\in B$,
Lemma \ref{lm31} and \eqref{eq44} gives that $h(n)\geq h(C)>1$.

If $2$ divides $n$ and $3$ does not belong to $C$, then, by Lemma \ref{lm31}, we have
$h(n)\geq h(C\cup \{2, 3^2\})=h(\{2, 3^2\}) h(C)>(3/4)(13/9)>1$.

If $2\times 3^6$ divides $n$ and $3$ belongs to $C$, then
\eqref{eq44} yields that $h(C)>3$ and $h(n)\geq (3/4)h(C)>9/4>1$.

Thus, in any case, we have $h(n)>1$ or, equivalently,
$\sigma(n)>(\gamma(n))^2$, which contradicts to the assumption that $n$ satisfies \eqref{eq10}.
\end{proof}

Now it suffices to settle two cases: $e_0=1$ and $3^2\in N(L)$ or $e_0=1$ and $3^4\in B(L)$.

\begin{lem}\label{lm42}
If $e_0=1$ and $3^2\in N=N(L)$, then $N$ must have a cycle
or we must have $L=\{p_1, p_2\}$ and $p_1\in B(p_2)$ or $p_2\in B(p_1)$.
\end{lem}

\begin{proof}
Assume that $3^2\in N$, $N$ is acyclic and,
in the case (b) with $e_1=e_2=1$, $p_1\not\in B(p_2)$ and $p_2\not\in B(p_1)$.
Since $3^2$ belongs to $N$, $3^2\rightarrow 13$ also belongs to $N$.
If $13\in M=M(L)$, then $3^2\rightarrow 13^2\rightarrow 3$,
which contradicts to the assumption that $N$ is acyclic.
Thus, $13^1\in L$.
Now we may assume that $p_1=13$.
We see that $p_2\in L$ and $p_2\equiv 1\mathmod{4}$ since $p_1\equiv 1\mathmod{4}$.
Hence, $13\rightarrow 7^e$ divides $N$.

We see that $e\geq 2$ must be even since $2^3\mid (7+1)$.
If $7^2\mid\mid N$, then $13\rightarrow 7^2\rightarrow 3^2\rightarrow 13$,
contrary to the assumption that $N$ is acyclic.
Thus, $e\geq 4$.

If $7^e\not\in B(p_2)$, then, applying Lemma \ref{lm40}, we have
\begin{equation}
h(n)\geq h(\{2, 3^2, 13, 7^e\}) h(C(p_2))>h(C(p_2))>1,
\end{equation}
which is a contradiction.
Similarly, if $L=\{p_1\}$, then $h(n)\geq h(\{2, 3^2, 13, 7^e\})$ $>1$,
which is a contradiction. 

Thus, we may assume that $7^e\in B(p_2)$.
If $e\geq 8$, then, Lemma \ref{lm40} gives that
\begin{equation}
h(\{2, 3^2, 13\}) h(C(p_2))>7^2 h(\{2, 3^2, 13\})>1,
\end{equation}
which is a contradiction again.

Assume that $7^4\in B(p_2)$, which immediately yields that $2801\in N(p_2)$.
If $p_2=2801$, then $p_2\rightarrow 3^2\rightarrow 13=p_1$,
contrary to the assumption that $p_2^{e_2}\not\in B(p_1)$.
Thus, $2801^2\in N(p_2)$ and $2801^2\rightarrow 37, 43, 4933$.

If $p_2=37$, then $p_3=19$ divides $n$.
If $p_2=4933$, then $p_3=2467$ divide $n$.
In both cases, if $p_3^2\mid\mid n$, then $p_2\rightarrow p_3^2\rightarrow 3^2\rightarrow 13=p_1$, which is impossible.
If $p_3^4\mid n$, then
\[h(n)\geq h(\{2, 3^2, 13, 7^4, 2801^2, 37, 19^4\})>1\]
or
\[h(n)\geq h(\{2, 3^2, 13, 7^4, 2801^2, 4933, 2467^4\})>1.\]
Hence, $p_2=37$ and $p_2=4933$ are both impossible.

If $37^2\in N(p_2)$, then $\sigma(37^2)=3\times 7\times 67$ and therefore $67\in N(p_2)$.
Since $67\equiv 3\mathmod{4}$, we have $p_2\neq 67$ and
$37^2 \rightarrow 67^2$.
But this implies that $3^3\mid \sigma(2\times 37^2\times 67^2)\mid \sigma(n)$,
which is a contradiction.

If $4933^2\in N(p_2)$, then $\sigma(4933^2)=3\times 127\times 193\times 331$ and therefore $p_2=193$,
since $p_3^2\in N(p_2)$ with $p_3=127, 193$ or $331$ would imply that $3^3\mid\sigma(2\times 4933^2\times p_3^2)$,
a contradiction.
Thus $p_3=97$ must divide $n$.
If $e_3=2$, then $3^3\mid\sigma(2\times 4933^2\times 97^2)\mid \sigma(n)$, which is impossible.
But, if $e_3\geq 4$, then
\[h(n)\geq h(\{2, 3^2, 13, 7^4, 2801^2, 4933^2, 193, 97^4\})>1,\]
which is a contradiction again.

If $43^2\in N(p_2)$, then $\sigma(43^2)=3\times 631$ and therefore $631\in N(p_2)$.
Since $631\equiv 3\mathmod{4}$, we must have $631^2\in N(p_2)$
and $3^3\mid\sigma(2\times 43^2\times 631^2)\mid \sigma(n)$, which is impossible.
Thus we see that $2801^f\not\in N(p_2)$ and therefore $7^4\not\in N(p_2)$.

Now we must have $7^6\in B(p_2)$.
$\sigma(7^6)=29\times 4733$ must divide $n$.
It is impossible that $p_2=29, 4733$ since this would imply that $p_2\rightarrow 3^2\rightarrow 13=p_1$.
If $29^2\in N(p_2)$, then, observing that $\sigma(29^2)=13\times 67$ and $\sigma(67^2)=3\times 7^2\times 31$,
we must have $29^2 \rightarrow 67^2 \rightarrow 31^2$.
However, this is impossible since $3^3\mid\sigma(2\times 67^2\times 31^2)$.

If $4733^2\in N(p_2)$, then, observing that $4733^2+4733+1=22406023\equiv 3\mathmod{4}$ is prime, we must have $22406023^2\in N(p_2)$.
If $22406023^2\rightarrow p_2$, then $p_2=1117$ or $p_2=606538249$.
However, neither of them can occur since $13^3\mid \sigma(3^2\times 22406023^2\times 1117)$
and $5^3\mid \sigma(606538249)$.
Hence, we must have $22406023^2\rightarrow p_3^2\in N(p_2)$ for some prime divisor $p_3\neq 3$ of $\sigma(22406023^2)$.
But, this is also impossible since $3^3\mid \sigma(2\times 22406023^2\times p_3^2)$.

Now we conclude that $7$ cannot divide $N$ and therefore $13$ cannot be in $L$.
Hence, $3^2$ cannot be in $N(L)$.
This proves the lemma.
\end{proof}

\begin{lem}\label{lm43}
If $e_0=1$ and $3^4\in B=B(L)$, then
$n=1782$, $N=N(L)$ must have a cycle or we must have $L=\{p_1, p_2\}$ and $p_1\in B(p_2)$ or $p_2\in B(p_1)$.
\end{lem}

\begin{proof}
Since $3^4\in B$, $p_1=11$ or $11^2\in N$.
If $p_1=11$, then $n=2\times 3^4\times p_1=1782$.
We note that if $n=n_0$ is a solution of \eqref{eq10}, then $n=kn_0$ with $k>1$ odd and $\gcd(k, n)=1$ can never be a solution of \eqref{eq10}.
Indeed, $h(n_0)=h(kn_0)=1$, then $h(k)=1$.
However, this is impossible since $n=1$ is the only odd solution of \eqref{eq10}.

Now we may assume that $11^2\in N$.
If $p_3^2\in N$ with $p_3=7$ or $19$, then $3^4\rightarrow 11^2\rightarrow p_3^2\rightarrow 3^4$ in $N$,
contrary to the assumption.
Thus, we must have $p_1=19$ and $7^4\mid n$.
Since $p_1\equiv 3\mathmod{8}$, we must have $L=\{p_1\}$.
Hence,
\[h(n)\geq h(\{2, 3^4, 11^2, 7^4, 19\})>1,\]
which is impossible again.
\end{proof}

\section{Cyclic cases}

In the previous section, we showed that,
if an integer $n$ of the form \eqref{eq11} satisfies \eqref{eq10}
and $L$ is the set of odd primes $p_i$ with $e_i=1$, then $N(L)$ must be cyclic
or we must have $L=\{p_1, p_2\}$ and $p_1\in B(p_2)$ or $p_2\in B(p_1)$.
In this section, we shall show that $M(L)$ must be acyclic and then
complete the proof of Theorems \ref{th11} and \ref{th12}.
We begin by showing that $M=M(L)$ cannot contain a cycle of length $\geq 3$.

\begin{lem}\label{lm51}
Assume that for there exists no arc $p_i\rightarrow p_j$ from $p_i\in L$ to $p_j\in N(L)$.
Then $M=M(L)$ cannot contain a cycle of length $\geq 3$.
\end{lem}
\begin{proof}
Assume that $q_i (i=1, 2, \ldots, l)$ is a cycle of length $l\geq 3$.
We see that $q_i\equiv 1\mathmod{3}$ for all $i$ except possibly one index $j$, for which $q_j=3$.
We must have $l=3$ and $q_j=3$ for some $j$ since otherwise we must have
$q_i\equiv 1\mathmod{3}$ for at least three $i$'s by Lemma \ref{lm31}
and $3^3\mid \prod_j \sigma(q_i^2)\mid n$, which is a contradiction.

Now we see that $3^2\rightarrow 13^2\rightarrow 61^2\rightarrow 3^2$ is a cycle in $M$ and $p_1=97\in L$.
Hence, $97\rightarrow 7^e$ must divide $n$ and,
observing that no more prime $p_i\equiv 1\mathmod{3}$ can satisfy $p_i^2\mid\mid N$ again, $e\geq 4$ must be even.
Moreover, we must have $e_0\geq 2$ since $3^3\mid \sigma(2\times 13^2\times 61^2)$.

If $L=\{p_1\}$ and $7^6$ divides $n$, then
\[h(n)\geq h(7^6)h(C(L))>h(\{3^2, 13^2, 61^2, 97, 7^6\})>1,\]
which is a contradiction.

If $L=\{p_1, p_2\}$ and $7^{10}$ divides $n$, then, since $N(p_2)$ is acyclic, Lemma \ref{lm40} gives
\begin{equation}
\begin{split}
h(n)\geq & h(\{3^2, 13^2, 61^2, 97\}) h(C(p_2)\cup \{7^{10}\}) \\
> & 7^3 h(\{3^2, 13^2, 61^2, 97\})>1.
\end{split}
\end{equation}

Now we must have $e=4, 6$ or $8$.
We can never have $97\rightarrow 7^8$ since $3^3\mid \sigma(13^2\times 61^2\times 7^8)$.
In both cases $e=4$ and $e=6$, we have a contradiction that $p^3\mid\sigma(n)=(\gamma(n))^2$ for some prime $p$
or $h(n)>1$ as follows:
\begin{itemize}
\item[A.] If $97\rightarrow 7^6$, then $\sigma(7^6)=29\times 4733$ divides $\sigma(n)$.
Moreover, we have $L=\{p_1, p_2\}$ with $7^6\in B(p_2)$ or $L=\{p_1, p_2\}$ with $7^6\not\in B(p_2)$
since it is impossible that $L=\{p_1\}$ and $7^6\mid n$ as seen above.
\item[A1.] If $7^6\in B(p_2)$, then $p_2=29$ or $4733$ or $7^6\rightarrow p_3^2\in N(p_2)$ with $p_3=29$ or $4733$.
\item[A1a.] If $p_2=29$ or $4733$, then $3^3\mid \sigma(13^2\times 61^2\times p_2)$.
\item[A1b.] We cannot have $7^6\rightarrow 29^2$ since $13^3\mid \sigma(3^2\times 61^2\times 29^2)$.
\item[A1c.] If $7^6\rightarrow 4733^2$ and $4733^2\in N(p_2)$, then $p_2=22406023$ or
$4733^2\rightarrow 22406023^2$.
Since $22406023\equiv 3\mathmod{4}$, we must have $4733^2\rightarrow 22406023^2$ and
therefore $3^3\mid \sigma(13^2\times 61^2\times 22406023^2)$.
\item[A2.] If $L=\{p_1, p_2\}$ and $7^6\not\in B(p_2)$, then, $N(p_2)$ has no cycle by the assumption
and Lemma \ref{lm40} gives
\[h(n)\geq h(\{97, 7^6\})h(C(p_2))>h(\{97, 7^6\})>1.\]
\item[B.] If $97\rightarrow 7^4$, then $7^4\rightarrow 2801^f$ for some integer $f>0$.
\item[B1.] If $f\equiv 1\mathmod{4}$, then $3^3\mid \sigma(13^2\times 61^2\times 2801)\mid \sigma(n)$.
\item[B2.] If $f\geq 6$ and $L=\{p_1, p_2\}$, then
\[h(n)\geq h(\{97, 7^4, 2801^f\}\cup C(p_2))>2801h(\{97, 7^4\})>1.\]
\item[B3.] If $f\geq 4$ and $L=\{p_1\}$, then
\[h(n)\geq h(\{97, 7^4, 2801^f\})>2801h(\{97, 7^4\})>1.\]
\item[B4.] If $f=4$ and $2801^4\not\in B(p_2)$, then
\[h(n)\geq h(\{97, 7^4, 2801^4\}\cup C(p_2))>h(\{97, 7^4, 2801^4\})>1.\]
\item[B5.] If $f=4, L=\{p_1, p_2\}$ and $2801^4\in B(p_2)$, then $q\in N(p_2)$ with $q=5$, $1956611$ or $6294091$.
\item[B5a.] If $p_2=5$, then $3^3\mid\sigma(13^2\times 61^2\times p_2)\mid \sigma(n)$.
\item[B5b.] If $5^2\in N(p_2)$, then $5^2\rightarrow 31^2$ but
$3^3\mid\sigma(13^2\times 61^2\times 31^2)\mid \sigma(n)$.
\item[B5c.] We cannot have $6294091^2\in N(p_2)$ since $3^3\mid\sigma(13^2\times 61^2\times 6294091^2)$.
\item[B5d.] If $1956611^2\in N(p_2)$, then
$\sigma(1956611^2)=277\times 95479\times 144751$.
If $1956611^2\rightarrow p_3^2$ with $p_3=277$, $95479$ or $144751$,
then $3^3\mid \sigma(13^2\times 61^2\times p_3^2)\mid \sigma(n)$.
If $p_2=277$, then $h(n)\geq h(\{97, 7^4, 2801^4, 277\})>1$.
\item[B6.] If $f=2$, then $\sigma(2801^2)=37\times 43\times 4933$ divides $\sigma(n)$.
\item[B6a.] We cannot have $2801^2\rightarrow 43^2$ since $3^3\mid\sigma(13^2\times 61^2\times 43^2)$.
\item[B6b.] If $2801^2\rightarrow 43^{e_3}, e_3\geq 6$, then
\[h(n)\geq h(\{97, 7^4, 43^{e_3}\}\cup C(p_2))>43 h(\{97, 7^4\})>1.\]
\item[B6c.] If $2801^2\rightarrow 43^4$ and $43^4\not\in B(p_2)$, then
\[h(n)\geq h(\{97, 7^4, 43^4\}\cup C(p_2))>h(\{97, 7^4, 43^4\})>1.\]
\item[B6d.] If $2801^2\rightarrow 43^4$ and $43^4 \in B(p_2)$, then $43^4\rightarrow 3500201^1$ or
$43^4\rightarrow 3500202^2$.
If $43^4\rightarrow 3500201^1$, then $3^3\mid \sigma(13^2\times 61^2\times 3500201)\mid\sigma(n)$.
If $43^4\rightarrow 3500201^2$, then 
$\sigma(3500201^2)$ $=13$ $\times 139$ $\times 28411$ $\times 238639$.
Since $q\equiv 3\mathmod{4}$ for $q=139$, $28411$ and $238639$, $q^2\in N(q_2)$ and
$3^3\mid \sigma(13^2\times 61^2\times q^2)\mid \sigma(n)$.
\end{itemize}

Thus we have a contradiction in any case.
This yields that $3^2\rightarrow 13^2\rightarrow 61^2\rightarrow 97$ is impossible.
Hence, we conclude that $M=M(L)$ cannot contain a cycle of length $\geq 3$, as stated in the lemma.
\end{proof}

Now a cycle in $M(L)$ must be of the form $p_i^2\leftrightarrow p_j^2$.
We may assume that $p_r^2\leftrightarrow p_{r+1}^2$ for some $r$.
In other words, we must have $p_r\mid\sigma(p_{r+1}^2)$ and $p_{r+1}\mid\sigma(p_r^2)$
for some primes $p_r, p_{r+1}\in M(L)$.

Lemma 2.6 of \cite{CT} shows that such $p_r, p_{r+1}$ must be two consecutive terms
of the binary recurrent sequence described in A101368 of OEIS.
This had already been proved by Mills \cite{Mil} and Chao \cite{Cha}.
However, this fact is not needed in our argument.
We only use the fact that, if $p_{r+1}>p_r>3$ and $p_r\leftrightarrow p_{r+1}$,
then $p_r\equiv p_{r+1}\equiv 1\mathmod{3}$ by Lemma \ref{lm32}.

We begin by proving that, we cannot have $p_r\leftrightarrow p_{r+1}$ if $p_{r+1}>p_r>3$.

\begin{lem}\label{lm52}
Assume that for there exists no arc $p_i\rightarrow p_j$ from $p_i\in L$ to $p_j\in N(L)$.
If $M=M(L)$ contains a cycle $p_r^2\leftrightarrow p_{r+1}^2$ of length two with $p_{r+1}>p_r$,
then $(p_r, p_{r+1})=(3, 13)$.
\end{lem}

\begin{proof}
We may assume that $p_r, p_{r+1}\in N(p_1)$.
Hence, there exists a vertex $q\in N(p_1)$ such that $p_r\rightarrow q$ or $p_{r+1}\rightarrow q$.
However, if $q\in M$, then, since $q\equiv p_{r+1}\equiv 1\mathmod{3}$,
we must have $3^3\mid \sigma(q^2 p_r^2 p_{r+1}^2)\mid \sigma(n)$, which is a contradiction.
Thus, we must have $q\in L$.

Now we obtain a directed multigraph $F$ by eliminating the arcs
$p_r\leftrightarrow p_{r+1}$ and $p_r$ or $p_{r+1}\rightarrow p_i$ with $p_i\in L$ from $C=C(L)$.
Then $F$ has two more sinks $p_r, p_{r+1}$ as well as sinks in $C(L)$.

Proceeding as in the proof of Lemma \ref{lm40}, we have
\begin{equation}
\prod_{p_i\in B=B(L)} \sigma(p_i^{e_i})=\left(\prod_{p_i\in B}\kappa_i\right)\prod_{p_j\in M, j\neq r, r+1}\lm_j^{s_j}
\prod_{p_i\in L\cup \{p_r, p_{r+1}\}}p_j^{2s_j},
\end{equation}
where
\begin{equation}
s_j=\sum_{\substack{q_0\rightarrow q_1\rightarrow \cdots \rightarrow q_k=p_j, \\
q_0\in B, q_1, \ldots, q_{k-1}\in N}} \frac{1}{2^k}.
\end{equation}

Let $f_i$ be the exponent $p_i^{f_i}\mid\mid \sigma(p_r^2 p_{r+1}^2)$ for $p_i\in L$.
We observe that $d_F^-(p_i)=2-f_i$ for $p_i\in L$, $d_F^-(p_r)=d_F^-(p_{r+1})=1$ and
$d_F^-(p_j)=2$ for any other vertex $p_j$ in $N$.
Hence,
\begin{equation}
s_j=t_j\sum_{\substack{q_0\rightarrow q_1\rightarrow \cdots \rightarrow q_k=p_j, \\
q_0\in B, q_1, \ldots, q_{k-1}\in N}} \frac{1}{\prod_{l=1}^k d_F^-(q_l)},
\end{equation}
where $t_j=(2-f_i)/2$ for $p_j\in L$, $1/2$ for $j=r, r+1$ and
$t_j=1$ for any other $j$ such that $p_j\in N$.
By Lemma \ref{lm21}, we have $s_j=t_j$ for any $j$ such that $p_j\in N$ and, as in Lemma \ref{lm40},
\begin{equation}
h(C)>\prod_{p_i\in B}\kappa_i^\frac{1}{2} p_i^{\frac{e_i}{2}-2}
\prod_{p_j\in M, j\neq r, r+1}\frac{\sqrt{\sigma(p_j^2)}}{p_j} \frac{p_r^\frac{1}{2}p_{r+1}^\frac{1}{2}}{p_1^\frac{f_1}{2} p_2^\frac{f_2}{2}}
>\sqrt{\frac{p_r p_{r+1}}{p_1^{f_1} p_2^{f_2}}}.
\end{equation}
If $p_r>3$, then we have $p_r\equiv p_{r+1}\equiv 1\mathmod{3}$
and $p_1^{f_1} p_2^{f_2}\leq$ \\ $\sigma(p_r^2 p_{r+1}^2)/(9p_r p_{r+1})$.
Moreover, we observe that $e_0\geq 2$ since $3^3\mid \sigma(2p_r^2 p_{r+1}^2)$.
Hence, we must have
\begin{equation}
h(n)\geq h(C)>\frac{3p_r p_{r+1}}{\sqrt{\sigma(p_r^2 p_{r+1}^2)}}>1,
\end{equation}
which is a contradiction.
Thus, we must have $(p_r, p_{r+1})=(3, 13)$.
\end{proof}

Now the only remaining case is $3^2\leftrightarrow 13^2 \rightarrow 61^1$.

\begin{lem}\label{lm53}
Assume that there exists no arc $p_i\rightarrow p_j$ from $p_i\in L$ to $p_j\in N=N(L)$.
Then, $3^2\leftrightarrow 13^2 \rightarrow 61^1$ is impossible.
\end{lem}

\begin{proof}
Assume that $3^2\leftrightarrow 13^2 \rightarrow 61^1$.
Then we immediately have $L=\{61\}$ or $L=\{61, p_2\}$ with $p_2\equiv 1\mathmod{4}$.
It is also clear that $61^1\rightarrow 31^{e_3}$.

If $e_3\geq 8$, then Lemma \ref{lm40} gives
\[h(n)\geq h(\{2, 3^2, 13^2, 61\}) h(C(p_2)\cup \{31^{e_3}\})>31^2 h(\{2, 61\})>1.\]
If $e_3\geq 4$ and $L=\{p_1\}$, then clearly we have
\[h(n)\geq h(\{2, 3^2, 13^2, 61, 31^{e_3}\})\geq h(\{2, 61, 31^4\})>1.\]
If $e_3\geq 4$, $L=\{p_1, p_2\}$ and $31^{e_3}\not\in B(p_2)$, then Lemma \ref{lm40} gives
\[h(n)\geq h(\{2, 3^2, 13^2, 61, 31^{e_3}\}) h(C(p_2))>h(\{2, 61, 31^4\})>1.\]
Thus, in these three cases, we are led to $h(n)>1$, which is a contradiction.
Hence, we must have (I) $L=\{p_1, p_2\}$, $e_3\in \{4, 6\}$ and $31^{e_3}\in B(p_2)$ or (II) $e_3=2$.
In both cases (I) and (II), we have a contradiction that $p^3\mid\sigma(n)=(\gamma(n))^2$ for some prime $p$
or $h(n)>1$ as follows:
\begin{itemize}
\item[I. A.] If $31^6\in B(p_2)$, then $p_2=917087137$ or $917087137^2\in N(p_2)$.
\item[I. A1.] In the case $p_2=917087137$, we observe that $p_4^{e_4}\rightarrow p_2$ for a prime $p_4\neq 31$.
\item[I. A1a.] If $e_4=2$, then $p_4\geq 20612597323$ and, since
$3^2, 13^2, 61^1, 31^6, p_2^1\not\in C(p_4)$
(we observe that $p_2^1\in C(p_4)$ implies that $N(p_2)$ must contain a cycle
$p_2\rightarrow \cdots \rightarrow p_4^2\rightarrow p_2$),
Lemma \ref{lm40} yields that
\[h(n)\geq h(\{2, 61, 31^6, p_2\})h(C(p_4))>p_4 h(\{2, 61, 31^6, p_2\})>1.\]
\item[I. A1b.] If $e_4>2$, then $h(C(p_2))>31p_4>31^2$ by Lemma \ref{lm40} and therefore
\[h(n)\geq h(\{2, 61\})h(C(p_2))>31^2h(\{2, 61\})>1.\]
\item[I. A. 2.] If $917087137^2\in N(p_2)$, then, since any prime factor of $\sigma(p_4^2)$
is $\equiv 3\mathmod{4}$, we must have $p_4^2\rightarrow p_5^2$ with $p_5=43, 4447, 38647$ or $38533987$,
which is impossible since $3^3\mid \sigma(13^2 p_4^2 p_5^2)$.
\item[I. B.] If $31^4\in B(p_2)$, then one of $5, 5^2, 11^2, 17351^2$ must belong to $N(p_2)$.
\item[I. B1.] If $p_2=5$, then $h(n)\geq h(\{2, 61, 31^4, 5\})>1$, a contradiction.
\item[I. B2.] We cannot have $5^2\in N(p_2)$ since $\sigma(5^2)=31\in B(p_2)$.
\item[I. B3.] If $11^2\in N(p_2)$, then $7^2\in N(p_2)$ or $19^2\in N(p_2)$.
Since $\sigma(7^2)=3\times 19$, we have $19^2\in N(p_2)$ in any case.
Now we must have $19^2\rightarrow 127^2\in N(p_2)$.
Thus, $3^3\mid\sigma(13^2\times 19^2\times 127^2)\mid\sigma(n)$, a contradiction.
\item[I. B4.] If $17351^2\in N(p_2)$, then $1063^2\in N(p_2)$ or $21787^2\in N(p_2)$.
\item[I. B4a.] If $1063^2\in N(p_2)$, then we must have $1063^2\rightarrow 377011^2\in N(p_2)$
and $3^3\mid\sigma(13^2\times 1063^2\times 377011^2)$, which is a contradiction.
\item[I. B4b.] If $21787^2\in N(p_2)$, then $p_2=5104249$ or $5104249^2\in N(p_2)$.
Neither of them is possible since $5^3\mid (5104249+1)$ and $3^3\mid\sigma(13^2\times 21787^2\times 5104249^2)$.
\item[II.] If $61\rightarrow 31^2$, then we must have $31^2\rightarrow 331^{e_3}$ for some $e_3$.
Since $331\equiv 3\mathmod{4}$, $p_2\neq 331$ and $e_3$ must be even.
\item[II. 1.] $e_3=2$ is impossible since $3^3\mid\sigma(13^2\times 31^2\times 331^2)$.
\item[II. 2.] If $e_3\geq 6$, $L=\{p_1, p_2\}$ and $331^{e_3}\in B(p_2)$, then Lemma \ref{lm40} gives
\[h(n)\geq h(\{2, 61\}) h(C(p_2)\cup \{331^{e_3}\})>331 h(\{2, 61\})>1.\]
\item[II. 3.] If $e_3\geq 4$ and $L=\{p_1\}$, then Lemma \ref{lm40} gives
\[h(n)\geq h(\{2, 61, 331^{e_3}\})\geq h(\{2, 61, 331^4\})>1.\]
\item[II. 4.] If $e_3\geq 4$, $L=\{p_1, p_2\}$ and $331^{e_3}\not\in B(p_2)$, then Lemma \ref{lm40} gives
\[h(n)\geq h(\{2, 61, 331^{e_3}\}) h(C(p_2))>h(\{2, 61, 331^4\})>1.\]
\item[II. 5.] If $e_3=4$, $L=\{p_1, p_2\}$ and $331^4\in B(p_2)$,
then $p_2=5, 37861, 63601$ or $331^4\rightarrow p_4^2\in N(p_2)$
with $p_4=37861$ or $63601$. (we see that since $\sigma(5^2)=31$, we cannot have $5^2\in N(p_2)$).
\item[II. 5a.] $331^4\rightarrow p_4^2\in N(p_2)$ is impossible since $3^3\mid\sigma(13^2\times 31^2\times p_4^2)$.
\item[II. 5b.] Assume that $p_2=5$, $37861$ or $63601$.
Since $L=\{p_1, p_2\}$ with $p_1=61$, we must have $37861\not\in L$ or $63601\not\in L$.
Thus, we see that $331^4\rightarrow p_4^{e_4}$ with $p_4=37861$ or $63601$ and $e_4\geq 4$.
Hence,
\[h(n)\geq h(\{2, 61, 331^4, p_2, p_4^{e_4}\})\geq h(\{2, 61, 331^4, 63601, 37861^4\})>1,\]
a contradiction again.
\end{itemize}

Thus we have a contradiction in any case.
This shows that $3^2\leftrightarrow 13^2 \rightarrow 61$ is impossible, as desired.
\end{proof}

Now we can easily prove Theorem \ref{th11}.
Let $n$ be an integer of the form \eqref{eq11} satisfying \eqref{eq10}
and $L$ be the set of odd primes $p_i$ such that $p_i\mid\mid n$.
If there exists no path between two vertices in $L$,
then, by Lemmas \ref{lm41}, \ref{lm42} and \ref{lm43}, $N(L)$ must have a cycle
but, by Lemmas \ref{lm51}, \ref{lm52} and \ref{lm53}, $M(L)$ cannot have a cycle.
Hence, $G(n)$ must have a path between two vertices in $L$ or a cycle in $N(L)$ containing a vertex in $L$.
This proves iii) of Theorem \ref{th12} and therefore Theorem \ref{th11}.

The remaining statements of Theorem \ref{th12} can be easily deduced from Lemma \ref{lm32}.
Let $g_1$ and $g_2$ be the number of primes $\equiv 1\mathmod{3}$ and $\nequiv 1\mathmod{3}$ in $M$
respectively.
i) and the former statement of ii) immediately follow from Lemma \ref{lm32}
and the fact that $3^3\nmid (\gamma(n))^2=\sigma(n)$.
Thus, $g_1\leq 2$.
If $p_i$ is a prime $\nequiv 1\mathmod{3}$ in $M$,
then $p_i^2\rightarrow p_j^2$ for some prime $p_j\equiv 1\mathmod{3}$ in $M$
or $p_i^2\rightarrow p_l$ for some prime $p_l\in L$.
Hence, we obtain $g_1+g_2\leq 2(g_1+\# L)$ and $g_2\leq g_1+2\# L\leq 2(1+\# L)$.
Now the latter statement of ii) follows.
This completes the proof of our theorems.

{}
\vskip 12pt
\end{document}